\documentclass[11pt]{article}
\usepackage[a4paper,margin=1in]{geometry}
\usepackage{amsmath,amssymb,amsthm,mathtools}
\usepackage{mathrsfs}
\usepackage{enumitem}
\usepackage{hyperref}
\usepackage[nameinlink]{cleveref}
\usepackage{booktabs}
\usepackage{array}
\usepackage{microtype}
\hypersetup{colorlinks=true,citecolor=blue,linkcolor=blue,urlcolor=blue,pdftitle={Variable-Length Markov Chains on Finite Quivers: Boundary-Window Identifiability, Exact Depth, and Local Rank Comparison},pdfauthor={Oleg Kiriukhin},pdfsubject={Variable-length Markov chains on finite quivers, boundary-window identifiability, exact depth, and local rank comparison},pdfkeywords={variable-length Markov chains, context-tree models, finite quivers, quiver-valued variable-length Markov chains, boundary-window identifiability, local memory depth, exact depth, local rank comparison, informative maps, minimal informative window, first-order identifiability}}
\newtheorem{theorem}{Theorem}[section]
\newtheorem{proposition}[theorem]{Proposition}
\newtheorem{lemma}[theorem]{Lemma}
\newtheorem{corollary}[theorem]{Corollary}
\theoremstyle{definition}
\newtheorem{definition}[theorem]{Definition}
\newtheorem{example}[theorem]{Example}
\newtheorem{remark}[theorem]{Remark}
\newtheorem{assumption}[theorem]{Assumption}
\newcommand{\cQ}{\mathcal{Q}}
\newcommand{\cV}{\mathcal{V}}
\newcommand{\cE}{\mathcal{E}}
\newcommand{\cS}{\mathcal{S}}

\newcommand{\RR}{\mathbb{R}}
\newcommand{\PP}{\mathbb{P}}
\newcommand{\rank}{\operatorname{rank}}
\newcommand{\Ker}{\operatorname{Ker}}
\newcommand{\im}{\operatorname{im}}

\title{Variable-Length Markov Chains on Finite Quivers: Boundary-Window Identifiability, Exact Depth, and Local Rank Comparison}
\author{Oleg Kiriukhin\\ City University of Hong Kong\\ \texttt{okiriukh@cityu.edu.hk}}
\date{April 2026}
\begin{document}
\maketitle

\begin{abstract}
Variable-length Markov chains on finite quivers provide a natural framework for context-dependent stochastic growth under incidence constraints.
I study quiver-valued variable-length Markov chains observed through finite boundary windows and develop a first-order theory of visible-depth identifiability in terms of stationary visible one-step transition laws and their restricted differentials on prescribed tangent blocks.

For visible depth $m$, the main object is the stationary one-step informative map $q_{\cQ}^{(m)}$.
In the edge-homogeneous regime, once the local visible support is fixed and the representation hypothesis is imposed, all admissible visible depths encode the same edge-level extension law and therefore have the same first-order rank.
In the exact-depth regime of context length $r$, the depth-$r$ boundary process is the canonical finite-state Markov chain, every smaller visible window is a deterministic truncation of that chain, and every coarser informative map factors $C^1$-smoothly through the depth-$r$ informative map on the relevant affine transition-array neighborhood.
In particular, the rank cannot increase beyond depth $r$.

After quotienting an arbitrary tangent block by the directions already invisible at depth $r$, I characterize strict coarse-depth loss exactly by coarse rank deficiency, equivalently by strict rank drop from depth $r$ to depth $m$ on the original tangent block.
I also give subspace-based and global selected-coordinate criteria, a global one-coordinate branching criterion, and an explicit depth-two example.
Under full fine-depth rank and strict coordinate-rank loss at every smaller depth, a global coordinate-rank theorem yields $m_*(T,\theta_0)=r$.
Reduced local coordinates remove stochastic redundancies, all first-order criteria are invariant under $C^1$ reparameterization, and the statistical and LAN consequences remain conditional on additional estimation and likelihood-level hypotheses.
\end{abstract}

\section{Introduction}
Variable-length Markov chains (VLMCs) and context-tree models describe processes whose one-step predictive law depends on a suffix of variable length rather than on a fixed memory depth; see, for example, \cite{buhlmann-wyner,rissanen,buhlmann-modelsel,machler-vlmc}.
I study a quiver-valued extension of this variable-length setting.
Replacing the alphabet by the edge set of a finite quiver imposes incidence constraints and turns the natural hidden state into an admissible path.
This yields quiver-valued variable-length Markov chains.

Often one observes only a finite boundary window.
The main deterministic problem is first-order boundary-window identifiability: which tangent directions in parameter space are detectable from a given visible depth, and which additional directions become detectable at larger visible depth?
Equivalently, I determine when a finite boundary window determines the relevant first-order information on local memory depth over a prescribed tangent block.
Thus the minimal informative window depends locally on $(\theta_0,T)$, where $\theta_0$ is the reference parameter and $T$ is the tangent block under consideration.

To my knowledge, this is the first systematic first-order theory of boundary-window identifiability and local rank comparison for quiver-valued VLMCs.

I prove four main deterministic theorems.
First, in the edge-homogeneous regime all admissible visible depths are locally equivalent once the local visible support is fixed and each relevant edge-extension pair is represented at the depths being compared.
Second, in the exact-depth regime every coarser visible law factors through the depth-$r$ law, so first-order rank cannot increase beyond depth $r$.
Third, after quotienting an arbitrary tangent block by the directions already invisible at depth $r$, failure of first-order local sufficiency is equivalent to strict coarse rank loss, or equivalently to strict rank drop from depth $r$ to depth $m$ on the original block.
Fourth, this intrinsic formulation yields both global selected-coordinate criteria and a global coordinate-rank theorem for recovery of the minimal informative window.

The statistical and local asymptotic normality (LAN) sections require additional estimation or likelihood hypotheses and record consequences of the deterministic theory only under those additional assumptions.
In particular, exact depth by itself yields rank monotonicity, but deterministic recovery of the true informative depth in the strengthened form proved later requires the additional coordinate-rank strict-loss input stated in \cref{thm:selected-criterion-global,thm:minimal-global}.
The Gaussian comparison statements require an additional likelihood-level factorization that bare LAN alone does not imply.

Parameter-geometric convention.
Throughout the paper the statistical model is parameterized by a $C^1$ chart on a finite-dimensional parameter manifold $\Theta$.
For local calculations I identify a neighborhood of the reference point $\theta_0\in\Theta$ with an open set in $\RR^d$, but every first-order statement is understood intrinsically on the tangent space $T_{\theta_0}\Theta$.
Thus a tangent block always denotes a linear subspace of $T_{\theta_0}\Theta$, and any Jacobian written in coordinates represents the differential of the corresponding map in the chosen chart.
Formal chart-invariance statements are recorded below as \cref{prop:chart-invariance}.

\paragraph{Relative-affine convention.}
When I say that a map is $C^1$ on a relative neighborhood inside an affine constraint set, I mean $C^1$ after choosing any affine chart on that constraint set.
Equivalently, after writing the affine set as $x_0+V$ with translation space $V$, the map becomes a $C^1$ map on an open neighborhood of $0$ in the Euclidean space $V$.
All such statements are independent of the chosen affine chart because affine changes of coordinates are smooth with smooth inverses.

\section{Model and informative maps}
\label{sec:model}

\subsection{Standing conventions}
Throughout the paper, all local statements are made after shrinking to neighborhoods on which the following objects are fixed: the relevant visible state spaces $\cS_m$, the admissible update maps $U_m$, the forced-zero pattern in each visible transition array, and the set of admissible edge-extension pairs that arise near $\theta_0$.
When I write that a visible word, update, or edge-extension pair \emph{arises near $\theta_0$}, I mean that it occurs with positive probability for every parameter in some sufficiently small neighborhood of $\theta_0$.
Stationary conditional probabilities are used only on neighborhoods where all conditioning events under discussion have stationary probabilities bounded away from $0$.
Whenever several depths are compared simultaneously, I tacitly work on sample paths whose current length is at least the largest depth under consideration, equivalently, one may fix any admissible initial path of that length.
When local coordinates are chosen, tangent spaces are identified with linear subspaces of $\RR^d$ through the selected chart.

\subsection{Finite quivers and right-growth models}
I write derivatives as restricted differentials $Dq(\theta_0)|_T:T\to E$, and a linear map is identified with a matrix only after bases are fixed explicitly.
The term \emph{full informative map} is reserved for the unreduced transition array, and \emph{reduced informative map} for its image under a reduced coordinate chart.
Let $\cQ=(\cV,\cE)$ be a finite quiver.
For each edge $e\in\cE$, write $s(e)$ and $t(e)$ for its source and target.
An admissible path is a finite word
\begin{equation*}
\omega=e_1\cdots e_k
\end{equation*}
of edges such that $t(e_i)=s(e_{i+1})$ for $1\le i<k$.
For $m\ge1$, the right boundary word of an admissible path of length at least $m$ is
\begin{equation*}
R_m(\omega)=e_{k-m+1}\cdots e_k.
\end{equation*}
Let $\Theta\subset\RR^d$ be open.
A one-sided right-growth model on $\cQ$ is a family $\{P_\theta:\theta\in\Theta\}$ under which the current admissible path grows by appending one admissible edge to the right at each discrete time step.
For an admissible visible context $\xi$ and an admissible appended edge $a$ with $s(a)=t(\xi)$, write
\begin{equation*}
\mu_\theta(a\mid\xi)
\end{equation*}
for the extension probability.
For each fixed $\xi$, these probabilities sum to $1$ over all admissible appended edges.

\begin{definition}[Edge-homogeneous regime]
The model is \emph{edge-homogeneous at $\theta_0$} if there exists a neighborhood $U\ni\theta_0$ such that for every $\theta\in U$, every admissible context $\xi$, and every admissible appended edge $a$,
\begin{equation}
\mu_\theta(a\mid\xi)=\mu_\theta(a\mid e),
\end{equation}
where $e$ is the last edge of $\xi$.

\end{definition}

\begin{definition}[Exact depth at $\theta_0$]
Fix $r\ge1$.
The model has \emph{exact depth $r$ at $\theta_0$} if there exists a neighborhood $U\ni\theta_0$ such that:
\begin{enumerate}[label=\textup{(\alph*)}]
\item for every $\theta\in U$, every admissible context $\xi$ of length at least $r$, and every admissible appended edge $a$, the quantity $\mu_\theta(a\mid\xi)$ depends only on the suffix $R_r(\xi)$,
\item for every integer $k$ with $1\le k<r$ there exist admissible contexts $\xi,\xi'$ of length at least $k$ and an appended edge $a$ which is admissible from both $\xi$ and $\xi'$ such that $R_k(\xi)=R_k(\xi')$ and the maps $\theta\mapsto\mu_\theta(a\mid\xi)$ and $\theta\mapsto\mu_\theta(a\mid\xi')$ are not identical on any neighborhood of $\theta_0$.
\end{enumerate}

\end{definition}

\begin{remark}
The exact-depth condition is local rather than merely pointwise.
Part \textup{(a)} imposes depth-$r$ dependence uniformly on a neighborhood of $\theta_0$, while part \textup{(b)} excludes any smaller visible depth from representing the same extension law on any neighborhood of $\theta_0$.
Thus I study local structural depth near $\theta_0$, not merely the value at a single parameter of a pointwise memory-depth functional.

\end{remark}

\begin{remark}
\label{rem:k0-excluded} The restriction to $1\le k<r$ in part \textup{(b)} is deliberate. The case $k=0$ would require a separate convention for the empty suffix and would not by itself ensure that a common appended edge is admissible from both contexts. Excluding $k=0$ keeps the comparison well posed and leaves unchanged the intended notion of positive structural memory depth.

\end{remark}

\subsection{Visible state spaces and informative maps}
Fix $\theta_0\in\Theta$.
For each depth $m\ge1$, let $\cS_m$ denote the set of admissible length-$m$ words that occur with positive probability for every parameter in some sufficiently small neighborhood of $\theta_0$.
Thus $\cS_m$ is a common local support, fixed after shrinking the neighborhood if necessary.
Write $Z_t^{(m)}\in\cS_m$ for the visible depth-$m$ boundary word.
Throughout, identities involving $Z_t^{(m)}$ or $Z_t^{(r)}$ are understood on the event that the current path length is at least the relevant depth.
Equivalently, one may fix any admissible initial path of length at least the largest depth under consideration, in which case all displayed formulas hold for every $t\ge0$.

\begin{definition}[Edge chain]
Assume the model is edge-homogeneous near $\theta_0$.
The associated \emph{edge chain} is the finite-state Markov chain on the locally constant set of admissible edges whose transition probability from $e$ to $a$ is $\mu_\theta(a\mid e)$ whenever $s(a)=t(e)$.

\end{definition}

\begin{assumption}[Local regularity for informative maps]
\label{ass:regularity} Fix a finite set of visible depths under consideration. There exists a neighborhood $U\ni\theta_0$ such that for every $\theta\in U$:
\begin{enumerate}[label=\textup{(\roman*)}]
\item the relevant hidden finite-state chain is well defined on a locally constant state space and with locally constant forced-zero pattern,
\item the nonzero extension probabilities are $C^1$ in $\theta$,
\item in the exact-depth regime the depth-$r$ chain, and in the edge-homogeneous regime the edge chain, is irreducible on that fixed support,
\item every visible state used in a conditional probability has strictly positive stationary probability, and these stationary masses are bounded away from $0$ on some smaller neighborhood of $\theta_0$,
\item every visible state space $\cS_m$ and every admissible update map $U_m$ used in the paper are locally constant on $U$.
\end{enumerate}

\end{assumption}

\begin{remark}
\Cref{ass:regularity} collects the local hypotheses required throughout the paper: support stability, local constancy of visible state spaces and update maps, $C^1$ dependence of the nonzero transition coordinates, irreducibility on the relevant finite support, and positivity of the visible stationary masses appearing in conditional laws.
The lower bound in \textup{(iv)} ensures that all stationary conditional probabilities entering $q_{\cQ}^{(m)}$ are defined on a common neighborhood and involve no vanishing denominators.
In later proofs I indicate explicitly which parts of \cref{ass:regularity} are used when needed.

\end{remark}

\begin{definition}[Visible informative maps]
For an admissible depth $m$ and visible states $y,y'\in\cS_m$, define the stationary one-step visible transition law by
\begin{equation}
q_{y,y'}^{(m)}(\theta):=P_{\theta,\mathrm{stat}}(Z_{t+1}^{(m)}=y'\mid Z_t^{(m)}=y),
\end{equation}
whenever the conditioning event has positive stationary probability.
Flattening all state-pair coordinates gives the full informative map
\begin{equation}
q_{\cQ}^{(m)}(\theta)\in\RR^{|\cS_m|^2}.
\end{equation}
Forced zeros and row-sum constraints are retained in this full map.

\end{definition}

\begin{proposition}[Regularity of informative maps in the two structural regimes]
\label{prop:q-regularity} Assume \cref{ass:regularity}. Fix a visible depth $m$.
\begin{enumerate}[label=\textup{(\roman*)}]
\item If the model has exact depth $r\ge m$ at $\theta_0$, then $q_{\cQ}^{(m)}$ is well defined and $C^1$ near $\theta_0$.
\item If the model is edge-homogeneous near $\theta_0$, then $q_{\cQ}^{(m)}$ is well defined and $C^1$ near $\theta_0$.
\end{enumerate}

\end{proposition}

\begin{remark}
The proof of the exact-depth part uses results established later in \cref{sec:exact}.
The present regularity statement follows immediately from the factorization theorem proved there.

\end{remark}

\begin{proof}
In the exact-depth regime, \cref{prop:augmentation} identifies the depth-$r$ boundary process with a finite-state Markov chain on the common local support $\cS_r$, and \cref{cor:factorization} shows that for $\theta$ near $\theta_0$ one has
\begin{equation*}
q_{\cQ}^{(m)}(\theta)=G_{r,m}(q_{\cQ}^{(r)}(\theta))
\end{equation*}
for a $C^1$ map $G_{r,m}$ defined on a relative neighborhood of the affine transition family.
Since the depth-$r$ informative map is just the flattened depth-$r$ transition matrix, its coordinates are $C^1$ by assumption on the nonzero extension probabilities, and therefore so is $q_{\cQ}^{(m)}$.
In the edge-homogeneous regime, \cref{prop:edge-markov,cor:edge-stationary} show that every coordinate of $q_{\cQ}^{(m)}(\theta)$ is either a forced zero or a coordinate of the edge-extension law $\theta\mapsto\mu_\theta(a\mid e)$.
Those nonzero coordinates are $C^1$ by \cref{ass:regularity}, so the full informative map is $C^1$.
In both regimes the conditioning events defining the stationary visible transition laws are well defined because the corresponding visible stationary probabilities are positive by \cref{ass:regularity}.

\end{proof}

\begin{definition}[Minimal informative window]
Let $T\subset T_{\theta_0}\Theta$ be a linear subspace.
Define
\begin{equation}
m_*(T,\theta_0):=\inf\Bigl\{m\ge1:\rank\bigl(Dq_{\cQ}^{(m)}(\theta_0)|_T\bigr)=\dim
T\Bigr\}\in\mathbb N\cup\{\infty\},
\end{equation}
with the convention that $m_*(T,\theta_0)=\infty$ if no depth attains full column rank $\dim T$.

\end{definition}

\begin{proposition}[Chart invariance of first-order criteria]
\label{prop:chart-invariance} Let $\psi:(\widetilde\Theta,\widetilde\theta_0)\to(\Theta,\theta_0)$ be a $C^1$ local reparameterization with invertible derivative at $\widetilde\theta_0$, and let $\widetilde q_{\cQ}^{(m)}:=q_{\cQ}^{(m)}\circ\psi$. Identify a tangent block $\widetilde T\subset T_{\widetilde\theta_0}\widetilde\Theta$ with $T:=D\psi(\widetilde\theta_0)\widetilde T\subset T_{\theta_0}\Theta$. Then for every visible depth $m$,
\begin{equation}
D\widetilde q_{\cQ}^{(m)}(\widetilde\theta_0)|_{\widetilde T} = Dq_{\cQ}^{(m)}(\theta_0)|_T\circ
D\psi(\widetilde\theta_0)|_{\widetilde T}.
\end{equation}
Consequently, the rank of the restricted derivative, kernel inclusions between visible depths on corresponding tangent blocks, first-order local sufficiency, and the minimal informative window $m_*(T,\theta_0)$ are invariant under $C^1$ reparameterization.

\end{proposition}

\begin{proof}
The displayed identity is just the chain rule.
Since $D\psi(\widetilde\theta_0)|_{\widetilde T}:\widetilde T\to T$ is a linear isomorphism, right-composition with it does not change rank and identifies kernels.
In particular, full-column-rank of the restricted derivative is preserved under the change of coordinates, so the defining property of $m_*(T,\theta_0)$ is unchanged.
The stated invariance properties follow immediately.

\end{proof}

\begin{definition}[First-order local sufficiency]
For $m<r$ and a tangent block $T\subset T_{\theta_0}\Theta$, the depth-$m$ window is \emph{first-order locally sufficient relative to depth $r$ on $T$} if
\begin{equation}
\Ker\bigl(Dq_{\cQ}^{(m)}(\theta_0)|_T\bigr) \subset \Ker\bigl(Dq_{\cQ}^{(r)}(\theta_0)|_T\bigr).
\end{equation}

\end{definition}

\begin{lemma}[Equivalent linear factorization]
\label{lem:suff-equivalent} All kernels and images below are taken for the restricted differentials on the tangent block $T$. Fix $m<r$ and a tangent block $T\subset T_{\theta_0}\Theta$ identified in local coordinates with a subspace of $\RR^d$. The following are equivalent:
\begin{enumerate}[label=\textup{(\roman*)}]
\item the depth-$m$ window is first-order locally sufficient relative to depth $r$ on $T$,
\item there exists a linear map
\begin{equation*}
A:\im\bigl(Dq_{\cQ}^{(m)}(\theta_0)|_T\bigr)\to \im\bigl(Dq_{\cQ}^{(r)}(\theta_0)|_T\bigr)
\end{equation*}
with
\begin{equation}
Dq_{\cQ}^{(r)}(\theta_0)|_T=A\circ Dq_{\cQ}^{(m)}(\theta_0)|_T.
\end{equation}
\end{enumerate}

\end{lemma}

\begin{proof}
Condition \textup{(ii)} implies \textup{(i)} immediately.
Conversely, assume \textup{(i)}.
If
\begin{equation*}
\nu=Dq_{\cQ}^{(m)}(\theta_0)h\in\im\bigl(Dq_{\cQ}^{(m)}(\theta_0)|_T\bigr)
\end{equation*}
for some $h\in T$, define
\begin{equation*}
A(\nu):=Dq_{\cQ}^{(r)}(\theta_0)h.
\end{equation*}
This is well defined because two representatives differ by an element of $\Ker(Dq_{\cQ}^{(m)}(\theta_0)|_T)$, which lies in $\Ker(Dq_{\cQ}^{(r)}(\theta_0)|_T)$ by assumption.
Linearity is immediate.

\end{proof}

\section{Exact depth and deterministic truncation}
\label{sec:exact}

\begin{definition}[Deterministic boundary update]
Fix a depth $\ell\ge1$.
For $y\in\cS_\ell$ and an admissible appended edge $a$ from the terminal vertex of $y$, let $U_\ell(y,a)$ be the visible depth-$\ell$ state obtained by appending $a$ to $y$ and truncating back to length $\ell$.

\end{definition}

\begin{proposition}[Pathwise truncation identity]
\label{prop:projection-general} Let $m<r$. Whenever both boundary words are defined on a sample path,
\begin{equation}
Z_t^{(m)}=\Pi_{r,m}(Z_t^{(r)}) \qquad (t\ge0),
\end{equation}
where $\Pi_{r,m}$ sends a length-$r$ word to its suffix of length $m$.

\end{proposition}

\begin{proof}
Both variables are suffixes of the same current path.
Taking the suffix of length $m$ of the suffix of length $r$ yields the suffix of length $m$ of the original path.

\end{proof}

\begin{proposition}[Depth-$r$ Markov representation]
\label{prop:augmentation} Assume exact depth $r$ at $\theta_0$. Then, for every $\theta$ sufficiently near $\theta_0$, the process $Z^{(r)}=(Z_t^{(r)})_{t\ge0}$ is a time-homogeneous first-order Markov chain on the common local support $\cS_r$.

\end{proposition}

\begin{proof}
Fix $\theta$ near $\theta_0$.
Let $\mathcal F_t$ be the sigma-field generated by the path history up to time $t$.
By exact depth, the conditional law of the next appended edge given $\mathcal F_t$ depends only on the current suffix $R_r$ of the path, hence only on $Z_t^{(r)}$.
By \cref{ass:regularity}\textup{(v)}, the update rule $U_r$ is fixed on the neighborhood under consideration.
If the appended edge is $a$, then the next boundary state is $U_r(Z_t^{(r)},a)$.
This is the Markov property, and time-homogeneity follows because the extension rule does not depend on $t$.

\end{proof}

\begin{lemma}[Smooth stationary law]
\label{lem:stationary} Let $P_\theta$ be a $C^1$ family of irreducible stochastic matrices on a fixed finite state space. Then the stationary distribution $\pi_\theta$ depends $C^1$-smoothly on $\theta$.

\end{lemma}

\begin{proof}
Let $n$ be the number of states, let $\mathbf 1\in\RR^n$ be the all-ones column vector, and define
\begin{equation*}
A_\theta:=I-P_\theta+\mathbf 1\mathbf 1^\top.
\end{equation*}
I show that $A_\theta$ is invertible.
Let $x\in\RR^n$ satisfy $A_\theta x=0$.
Left-multiplying by any stationary row vector $\pi_\theta$ of $P_\theta$ gives
\begin{equation*}
0=\pi_\theta A_\theta x =\pi_\theta(I-P_\theta)x+\pi_\theta\mathbf 1\mathbf 1^\top x
=(\pi_\theta\mathbf 1)(\mathbf 1^\top x) =\mathbf 1^\top x,
\end{equation*}
because $\pi_\theta P_\theta=\pi_\theta$ and $\pi_\theta\mathbf 1=1$.
Hence $\mathbf 1^\top x=0$, and the equation $A_\theta x=0$ reduces to
\begin{equation*}
(I-P_\theta)x=0.
\end{equation*}
Thus $P_\theta x=x$.
For an irreducible finite-state stochastic matrix, the eigenspace for eigenvalue $1$ is one-dimensional and is spanned by $\mathbf 1$.
Therefore $x=c\mathbf 1$ for some scalar $c$.
Since $\mathbf 1^\top x=0$, necessarily $c=0$, so $x=0$.
Hence $A_\theta$ is invertible.
Now let $\pi_\theta$ denote the stationary row vector.
Since $\pi_\theta(I-P_\theta)=0$ and $\pi_\theta\mathbf 1=1$, one has
\begin{equation*}
\pi_\theta A_\theta = \pi_\theta(I-P_\theta)+\pi_\theta\mathbf 1\mathbf 1^\top = \mathbf 1^\top.
\end{equation*}
Therefore
\begin{equation*}
\pi_\theta=\mathbf 1^\top A_\theta^{-1}.
\end{equation*}
Because $\theta\mapsto A_\theta$ is $C^1$ and matrix inversion is $C^1$ on the open set of invertible matrices, the map $\theta\mapsto \pi_\theta$ is $C^1$ as claimed.

\end{proof}

\begin{lemma}[Local stability of irreducibility and visible positivity]
\label{lem:local-stability} Assume exact depth $r$ at $\theta_0$ together with \cref{ass:regularity}. Let $\mathcal A_r$ be the affine subspace of depth-$r$ transition arrays satisfying the forced-zero and row-sum constraints, and let $p_0=q_{\cQ}^{(r)}(\theta_0)\in\mathcal A_r$. Then there exists a relative neighborhood $\mathcal U_r\subset\mathcal A_r$ of $p_0$ such that every stochastic matrix represented by $p\in\mathcal U_r$ is irreducible on the fixed support, and for every visible state $y\in\cS_m$ used in the depth-$m$ informative map one has
\begin{equation}
\sum_{z\in\Pi_{r,m}^{-1}(y)}\pi(p)(z)>0,
\end{equation}
where $\pi(p)$ denotes the stationary law of the chain with transition matrix $p$.

\end{lemma}

\begin{proof}
At $p_0$ the finite-state chain is irreducible on a fixed support.
Choose, for each ordered pair of states, a directed path with strictly positive transition probabilities.
These paths use only coordinates allowed by the fixed forced-zero pattern from \cref{ass:regularity}, so the same coordinates remain available throughout the affine family $\mathcal A_r$.
Positivity of the finitely many entries occurring on these paths persists on a sufficiently small relative neighborhood inside $\mathcal A_r$, so irreducibility persists on that fixed support.
The visible stationary masses are continuous functions of $p$ and are positive at $p_0$ by \cref{ass:regularity}, positivity therefore persists after shrinking the same neighborhood.

\end{proof}

\begin{corollary}[Factorization through the depth-$r$ chain]
\label{cor:factorization} Assume exact depth $r$ at $\theta_0$ and \cref{ass:regularity}. Fix $m<r$, let $\mathcal A_r$ be the affine subspace of arrays in $\RR^{|\cS_r|^2}$ satisfying the forced-zero and row-sum constraints for depth $r$, and let $\mathcal A_m$ be the corresponding affine subspace at depth $m$. Then there exists a relative neighborhood $\mathcal U_r\subset\mathcal A_r$ of $q_{\cQ}^{(r)}(\theta_0)$ and a $C^1$ map
\begin{equation*}
G_{r,m}:\mathcal U_r\to \mathcal A_m
\end{equation*}
such that
\begin{equation}
q_{\cQ}^{(m)}(\theta)=G_{r,m}(q_{\cQ}^{(r)}(\theta))
\end{equation}
for all $\theta$ near $\theta_0$.
Consequently,
\begin{equation*}
\rank\bigl(Dq_{\cQ}^{(m)}(\theta_0)|_T\bigr) \le \rank\bigl(Dq_{\cQ}^{(r)}(\theta_0)|_T\bigr)
\end{equation*}
for every tangent block $T\subset T_{\theta_0}\Theta$.

\end{corollary}

\begin{proof}
By \cref{prop:augmentation} and \cref{ass:regularity}\textup{(i),(v)}, for every $\theta$ near $\theta_0$ the process $Z^{(r)}$ is a finite-state Markov chain on the fixed state space $\cS_r$ with transition matrix
\begin{equation*}
p_\theta(z,z'):=P_{\theta,\mathrm{stat}}(Z_{t+1}^{(r)}=z'\mid Z_t^{(r)}=z), \qquad z,z'\in\cS_r.
\end{equation*}
In the exact-depth regime, the depth-$r$ informative map is exactly the flattened transition matrix, so $q_{\cQ}^{(r)}(\theta)=p_\theta\in\mathcal A_r$.
By \cref{lem:local-stability}, after shrinking to a suitable relative neighborhood $\mathcal U_r\subset\mathcal A_r$ of $q_{\cQ}^{(r)}(\theta_0)$, every array $p\in\mathcal U_r$ determines an irreducible stochastic matrix on the fixed support and every visible denominator used below is bounded away from $0$.
Shrinking the parameter neighborhood once more if necessary, I assume that
\begin{equation*}
q_{\cQ}^{(r)}(\theta)\in\mathcal U_r
\end{equation*}
for all parameters under consideration.
By the relative-affine convention stated in the introduction, the relative neighborhood $\mathcal U_r\subset\mathcal A_r$ may be viewed in any affine chart on $\mathcal A_r$.
Applying \cref{lem:stationary} in such a chart, the stationary distribution of the chain with transition matrix $p\in\mathcal U_r$ depends $C^1$-smoothly on $p$ as a relative-affine variable.
Denote it by $\pi(p)$.
Fix $y,y'\in\cS_m$.
For $p\in\mathcal U_r$ define
\begin{equation*}
\Gamma_{y,y'}(p):= \frac{\sum_{z\in\Pi_{r,m}^{-1}(y)}\sum_{z'\in\Pi_{r,m}^{-1}(y')}\pi(p)(z)p(z,z')} {\sum_{z\in\Pi_{r,m}^{-1}(y)}\pi(p)(z)}.
\end{equation*}
The denominator is strictly positive on $\mathcal U_r$ by construction, and the numerator and denominator are $C^1$ in $p$, hence each coordinate $\Gamma_{y,y'}$ is $C^1$ on $\mathcal U_r$.
Collecting these coordinates defines a map $G_{r,m}:\mathcal U_r\to\RR^{|\cS_m|^2}$.
I show that $G_{r,m}(p)\in\mathcal A_m$.
If no pair $(z,z')\in\Pi_{r,m}^{-1}(y)\times\Pi_{r,m}^{-1}(y')$ can occur under one admissible update of the depth-$r$ chain, then every term in the numerator vanishes, so the corresponding coordinate is a forced zero.
\begin{equation*}
\sum_{y'\in\cS_m}\Gamma_{y,y'}(p) = \frac{\sum_{z\in\Pi_{r,m}^{-1}(y)}\pi(p)(z)\sum_{z'\in\cS_r}p(z,z')} {\sum_{z\in\Pi_{r,m}^{-1}(y)}\pi(p)(z)}=1,
\end{equation*}
because each row of $p$ sums to $1$.
Thus $G_{r,m}(p)$ satisfies the forced-zero and row-sum constraints defining $\mathcal A_m$.
Now take $p=q_{\cQ}^{(r)}(\theta)$.
By \cref{prop:projection-general}, one has $Z_t^{(m)}=\Pi_{r,m}(Z_t^{(r)})$ pathwise.
Therefore the event $\{Z_t^{(m)}=y\}$ is the disjoint union of the events $\{Z_t^{(r)}=z\}$ over $z\in\Pi_{r,m}^{-1}(y)$, and similarly for $y'$.
Using stationarity and the law of total probability, Thus
\begin{equation*}
P_{\theta,\mathrm{stat}}(Z_{t+1}^{(m)}=y',Z_t^{(m)}=y) =
\sum_{z\in\Pi_{r,m}^{-1}(y)}\sum_{z'\in\Pi_{r,m}^{-1}(y')} \pi(p)(z)p(z,z').
\end{equation*}
Dividing by
\begin{equation*}
P_{\theta,\mathrm{stat}}(Z_t^{(m)}=y)=\sum_{z\in\Pi_{r,m}^{-1}(y)}\pi(p)(z)>0
\end{equation*}
shows that $\Gamma_{y,y'}(p)=q_{y,y'}^{(m)}(\theta)$.
Hence
\begin{equation*}
q_{\cQ}^{(m)}(\theta)=G_{r,m}(q_{\cQ}^{(r)}(\theta))
\end{equation*}
for all $\theta$ near $\theta_0$.
Differentiating this identity at $\theta_0$ and restricting to a tangent block $T$ gives
\begin{equation*}
Dq_{\cQ}^{(m)}(\theta_0)|_T = DG_{r,m}(q_{\cQ}^{(r)}(\theta_0))\circ Dq_{\cQ}^{(r)}(\theta_0)|_T.
\end{equation*}
Therefore
\begin{equation*}
\rank\bigl(Dq_{\cQ}^{(m)}(\theta_0)|_T\bigr) \le \rank\bigl(Dq_{\cQ}^{(r)}(\theta_0)|_T\bigr),
\end{equation*}
as claimed.

\end{proof}

\section{The edge-homogeneous regime}
\label{sec:edge}

\begin{proposition}[Visible Markov property under edge homogeneity]
\label{prop:edge-markov} Assume the model is edge-homogeneous near $\theta_0$ and satisfies \cref{ass:regularity}. Fix a visible depth $m\ge1$. Then, for every $\theta$ near $\theta_0$, the visible process $Z^{(m)}$ is a time-homogeneous first-order Markov chain. More precisely, if $y\in\cS_m$ has last edge $e$ and $a$ is admissible from $e$, then
\begin{equation}
P_\theta\bigl(Z_{t+1}^{(m)}=U_m(y,a)\mid Z_t^{(m)}=y\bigr)=\mu_\theta(a\mid e).
\end{equation}

\end{proposition}

\begin{proof}
If $Z_t^{(m)}=y$, then the current path ends with the last edge $e$ of $y$.
By edge homogeneity the conditional law of the next appended edge depends only on $e$, not on the earlier past.
By \cref{ass:regularity}\textup{(v)}, the update map $U_m$ is fixed on the neighborhood under consideration.
Once the next edge is $a$, the next visible state is deterministically $U_m(y,a)$.

\end{proof}

\begin{corollary}[Stationary visible transitions under edge homogeneity]
\label{cor:edge-stationary} Under the hypotheses of \cref{prop:edge-markov}, if the process is started in stationarity, then for every visible state $y\in\cS_m$ with last edge $e$ and every admissible appended edge $a$ from $e$,
\begin{equation}
q_{y,U_m(y,a)}^{(m)}(\theta)=\mu_\theta(a\mid e)
\end{equation}
for all $\theta$ near $\theta_0$.

\end{corollary}

\begin{proof}
This is the stationary form of the transition identity from \cref{prop:edge-markov}.

\end{proof}

\begin{theorem}[Homogeneous windows are locally equivalent]
\label{thm:homogeneous} Assume edge homogeneity near $\theta_0$ and \cref{ass:regularity}. Let $m,n\ge1$ be admissible depths. Assume moreover the following representation hypothesis: every admissible edge-extension pair $(e,a)$ arising near $\theta_0$ is represented at both depths, in the sense that for each such pair there exist states $y\in\cS_m$ and $\widetilde y\in\cS_n$ whose last edge is $e$ and for which appending $a$ yields $U_m(y,a)$ and $U_n(\widetilde y,a)$, respectively. Then there exist fixed linear maps
\begin{equation*}
G_{m\to n}:\RR^{|\cS_m|^2}\to\RR^{|\cS_n|^2}, \qquad G_{n\to m}:\RR^{|\cS_n|^2}\to\RR^{|\cS_m|^2},
\end{equation*}
such that
\begin{equation}
q_{\cQ}^{(n)}(\theta)=G_{m\to n}(q_{\cQ}^{(m)}(\theta)), \qquad q_{\cQ}^{(m)}(\theta)=G_{n\to
m}(q_{\cQ}^{(n)}(\theta))
\end{equation}
for all $\theta$ near $\theta_0$.
Consequently,
\begin{equation*}
\rank\bigl(Dq_{\cQ}^{(m)}(\theta_0)|_T\bigr)=\rank\bigl(Dq_{\cQ}^{(n)}(\theta_0)|_T\bigr)
\end{equation*}
for every tangent block $T\subset T_{\theta_0}\Theta$.

\end{theorem}

\begin{remark}
The representation hypothesis in \cref{thm:homogeneous} is a genuine structural assumption.
It is automatic, for example, when every admissible last edge appears in at least one visible state at each depth under consideration and every admissible update from that edge remains visible after truncation.
Without it, equal-rank conclusions can fail simply because one visible depth omits coordinates encoding some edge-extension pair.

\end{remark}

\begin{remark}
I prove the theorem using explicit coordinate-copy and coordinate-selection maps.
In particular, it does not require the realized set $\{\rho(\theta):\theta\text{ near }\theta_0\}$ to span the ambient edge-law space $\RR^{|\mathcal I|}$.

\end{remark}

\begin{proof}
After shrinking the neighborhood of $\theta_0$ if necessary, fix the visible state spaces, the admissible update patterns, and the finite set $\mathcal I$ of admissible edge-extension pairs $(e,a)$ arising near $\theta_0$.
Define the edge-level vector
\begin{equation*}
\rho(\theta):=\bigl(\mu_\theta(a\mid e)\bigr)_{(e,a)\in\mathcal I}\in\RR^{|\mathcal I|}.
\end{equation*}
By \cref{cor:edge-stationary}, each coordinate of $q_{\cQ}^{(m)}(\theta)$ is either a forced zero or exactly one coordinate of $\rho(\theta)$.
Since the forced-zero pattern and admissible updates are fixed on the chosen neighborhood, there exists a fixed linear map
\begin{equation*}
F_m:\RR^{|\mathcal I|}\to\RR^{|\cS_m|^2}
\end{equation*}
such that
\begin{equation*}
q_{\cQ}^{(m)}(\theta)=F_m\rho(\theta) \qquad\text{for all $\theta$ near $\theta_0$.}
\end{equation*}
Likewise there exists a fixed linear map
\begin{equation*}
F_n:\RR^{|\mathcal I|}\to\RR^{|\cS_n|^2}
\end{equation*}
with
\begin{equation*}
q_{\cQ}^{(n)}(\theta)=F_n\rho(\theta).
\end{equation*}
I use the representation hypothesis to recover $\rho(\theta)$ from either visible depth by fixed coordinate-selection maps.
For each $(e,a)\in\mathcal I$, choose a state $y_{e,a}\in\cS_m$ and a state $\widetilde y_{e,a}\in\cS_n$ as in the statement, so that the distinguished transitions
\begin{equation*}
y_{e,a}\to U_m(y_{e,a},a), \qquad \widetilde y_{e,a}\to U_n(\widetilde y_{e,a},a)
\end{equation*}
record the same edge-level quantity $\mu_\theta(a\mid e)$.
Define linear coordinate-selection maps
\begin{equation*}
(H_mx)_{(e,a)}:=x_{y_{e,a},U_m(y_{e,a},a)}, \qquad (H_nx)_{(e,a)}:=x_{\widetilde
y_{e,a},U_n(\widetilde y_{e,a},a)}.
\end{equation*}
Then \cref{cor:edge-stationary} gives, for every $\theta$ near $\theta_0$,
\begin{equation*}
H_m(q_{\cQ}^{(m)}(\theta))=\rho(\theta), \qquad H_n(q_{\cQ}^{(n)}(\theta))=\rho(\theta).
\end{equation*}
Hence
\begin{equation*}
q_{\cQ}^{(n)}(\theta)=F_nH_m(q_{\cQ}^{(m)}(\theta)), \qquad
q_{\cQ}^{(m)}(\theta)=F_mH_n(q_{\cQ}^{(n)}(\theta)).
\end{equation*}
Thus the required factorizations hold with
\begin{equation*}
G_{m\to n}:=F_nH_m, \qquad G_{n\to m}:=F_mH_n.
\end{equation*}
Differentiating at $\theta_0$ and restricting to $T$ yields
\begin{equation*}
Dq_{\cQ}^{(n)}(\theta_0)|_T=G_{m\to n}\circ Dq_{\cQ}^{(m)}(\theta_0)|_T, \qquad
Dq_{\cQ}^{(m)}(\theta_0)|_T=G_{n\to m}\circ Dq_{\cQ}^{(n)}(\theta_0)|_T.
\end{equation*}
Hence each restricted derivative factors through the other.
The first identity gives
\begin{equation*}
\rank\bigl(Dq_{\cQ}^{(n)}(\theta_0)|_T\bigr) \le \rank\bigl(Dq_{\cQ}^{(m)}(\theta_0)|_T\bigr),
\end{equation*}
and the second gives the reverse inequality.
Therefore the two ranks coincide.

\end{proof}

\section{Reduced local coordinates}
\label{sec:reduced}
Reduced coordinates remove affine stochastic redundancies from the full transition array, and the rank statements are invariant under this passage.

\begin{definition}[Reduced coordinate chart]
Fix a visible depth $\ell$ and a neighborhood of $q_{\cQ}^{(\ell)}(\theta_0)$.
A \emph{reduced coordinate chart} for the visible transition family at depth $\ell$ is a linear map
\begin{equation*}
L_\ell:\RR^{|\cS_\ell|^2}\to\RR^{N_\ell}
\end{equation*}
whose restriction to the affine subspace of transition arrays satisfying the forced-zero and row-sum constraints is injective.
The reduced informative map is
\begin{equation*}
\bar q_{\cQ}^{(\ell)}:=L_\ell\circ q_{\cQ}^{(\ell)}.
\end{equation*}

\end{definition}

\begin{lemma}[Affine reconstruction from reduced coordinates]
\label{lem:affine-reconstruction} Fix a visible depth $\ell$ and let $\mathcal A_\ell$ be the affine subspace of arrays satisfying the forced-zero and row-sum constraints. If $L_\ell$ is a reduced coordinate chart, then $L_\ell(\mathcal A_\ell)$ is an affine subspace of $\RR^{N_\ell}$ and the inverse map
\begin{equation*}
(L_\ell|_{\mathcal A_\ell})^{-1}:L_\ell(\mathcal A_\ell)\to\mathcal A_\ell
\end{equation*}
is affine.
In particular, every full transition coordinate on $\mathcal A_\ell$ is an affine function of the reduced coordinates.

\end{lemma}

\begin{proof}
Write $\mathcal A_\ell=x_0+V_\ell$, where $V_\ell$ is the translation space.
Since $L_\ell$ is linear, its image of $\mathcal A_\ell$ is affine.
Injectivity on $\mathcal A_\ell$ implies injectivity on $V_\ell$, hence $L_\ell|_{V_\ell}$ is a linear bijection onto its image.
The inverse on the affine set is therefore affine.

\end{proof}

\begin{lemma}[Rank invariance under reduced coordinates]
\label{lem:reduced-rank} For every visible depth $\ell$ and tangent block $T\subset T_{\theta_0}\Theta$,
\begin{equation}
\rank\bigl(D\bar q_{\cQ}^{(\ell)}(\theta_0)|_T\bigr) = \rank\bigl(Dq_{\cQ}^{(\ell)}(\theta_0)|_T\bigr).
\end{equation}

\end{lemma}

\begin{proof}
The image of $Dq_{\cQ}^{(\ell)}(\theta_0)$ lies in the translation space $V_\ell$ of the affine constraint set.
Since $L_\ell$ is injective on that translation space, composing $Dq_{\cQ}^{(\ell)}(\theta_0)|_T$ with $L_\ell$ does not change rank.

\end{proof}

\begin{remark}
Statements such as ``the only first-order varying coordinates'' are to be interpreted in reduced coordinates.
For the full transition array this is generally false because stochastic rows satisfy linear relations.

\end{remark}

\section{A branching example for exact depth and strict loss}

\begin{example}[Depth-two branching above a visible edge]
\label{ex:branching} Consider the quiver with vertices $\{0,1,2,3\}$ and edges
\begin{equation*}
b:0\to1,\qquad c:2\to1,\qquad a:1\to3,\qquad d:3\to0,\qquad e:3\to2.
\end{equation*}
Take a right-context model of exact depth $2$ in which the only free parameters are
\begin{equation*}
\eta_1:=\mu_\theta(d\mid ba),\qquad \eta_2:=\mu_\theta(d\mid ca),
\end{equation*}
with complementary probabilities $\mu_\theta(e\mid ba)=1-\eta_1$ and $\mu_\theta(e\mid ca)=1-\eta_2$.
All other depth-two transitions are deterministic:
\begin{equation*}
ad\to db,\qquad db\to ba,\qquad ae\to ec,\qquad ec\to ca.
\end{equation*}

\end{example}

\begin{proposition}[Exact rank computation]
\label{prop:example-rank} In the setting of \cref{ex:branching}, with parameter block $(\eta_1,\eta_2)\in(0,1)^2$:
\begin{enumerate}[label=\textup{(\alph*)}]
\item the depth-two chain on $\cS_2=\{ba,ad,db,ae,ec,ca\}$ is irreducible and has stationary distribution
\begin{equation*}
\pi(\eta_1,\eta_2)=\frac{1}{3(\eta_2+1-\eta_1)}(\eta_2,\eta_2,\eta_2,1-\eta_1,1-\eta_1,1-\eta_1),
\end{equation*}
\item in the reduced chart given by the free coordinates $\mu(d\mid ba)$ and $\mu(d\mid ca)$,
\begin{equation*}
\bar q_{\cQ}^{(2)}(\eta_1,\eta_2)=(\eta_1,\eta_2),
\end{equation*}
\item in the reduced depth-one chart with free coordinate $q_{a,d}^{(1)}$,
\begin{equation*}
\bar q_{\cQ}^{(1)}(\eta_1,\eta_2)=\frac{\eta_2}{\eta_2+1-\eta_1},
\end{equation*}
\item for every interior point $\theta_0=(\eta_1^0,\eta_2^0)$,
\begin{equation}
\rank D\bar q_{\cQ}^{(2)}(\theta_0)=2, \qquad \rank D\bar q_{\cQ}^{(1)}(\theta_0)=1.
\end{equation}
\end{enumerate}
if $D:=\eta_2^0+1-\eta_1^0$ and $h=(1-\eta_1^0,-\eta_2^0)$, then
\begin{equation}
D\bar q_{\cQ}^{(1)}(\theta_0)h=0, \qquad D\bar q_{\cQ}^{(2)}(\theta_0)h\neq0.
\end{equation}

\end{proposition}

\begin{proof}
The stationary equations imply $x_1=x_2=x_3$ and $x_4=x_5=x_6$ for the ordered state list $(ba,ad,db,ae,ec,ca)$.
The balance relation $x_1(1-\eta_1)=x_6\eta_2$ gives the ratio $x_1:x_6=\eta_2:(1-\eta_1)$, and normalization then yields the stated stationary law.
Part \textup{(b)} is immediate from the reduced depth-two chart.
For part \textup{(c)}, the visible depth-one state $a$ has hidden fiber $\{ba,ca\}$, so under stationarity
\begin{equation*}
P_{\theta,\mathrm{stat}}(ba\mid a)=\frac{\eta_2}{\eta_2+1-\eta_1}, \qquad
P_{\theta,\mathrm{stat}}(ca\mid a)=\frac{1-\eta_1}{\eta_2+1-\eta_1}.
\end{equation*}
Multiplying by the corresponding fine-depth probabilities of moving to $d$ yields
\begin{equation*}
q_{a,d}^{(1)}(\eta_1,\eta_2)=\frac{\eta_2}{\eta_2+1-\eta_1}.
\end{equation*}
Differentiating gives
\begin{equation*}
D\bar q_{\cQ}^{(2)}(\theta_0)=I_2, \qquad D\bar
q_{\cQ}^{(1)}(\theta_0)=\frac{1}{D^2}(\eta_2^0,1-\eta_1^0),
\end{equation*}
from which the rank statements and the kernel direction follow.

\end{proof}

\begin{corollary}[The depth-two example fits the single-coordinate strict-loss criterion]
\label{cor:example-fits} In the setting of \cref{ex:branching}, fix an interior point $\theta_0=(\eta_1^0,\eta_2^0)$ and let $T_0:=\RR^2$ be the natural two-dimensional parameter block. Then the pair consisting of the visible depth-one state $y=a$ and the appended edge $d$ satisfies the hypotheses of \cref{cor:single-coordinate-branching}. Consequently, the depth-one window is not first-order locally sufficient relative to depth two on $T_0$.

\end{corollary}

\begin{proof}
By \cref{prop:example-rank}, one has $\rank D\bar q_{\cQ}^{(2)}(\theta_0)=2$.
By \cref{lem:reduced-rank}, the same full-rank statement holds for the full derivative $Dq_{\cQ}^{(2)}(\theta_0)|_{T_0}$.
The selected depth-one coordinate is $q_{a,d}^{(1)}$, whose derivative is nonzero by \cref{prop:example-rank}.
It therefore remains to verify the factorization condition, equivalently the kernel inclusion, required in \cref{cor:single-coordinate-branching}.
Since the reduced depth-one informative map has only the single free coordinate $q_{a,d}^{(1)}$, every full depth-one coordinate is an affine function of $q_{a,d}^{(1)}$ by \cref{lem:affine-reconstruction}, passing to differentials shows that the full derivative $Dq_{\cQ}^{(1)}(\theta_0)|_{T_0}$ factors through the single-coordinate derivative $Dq_{a,d}^{(1)}(\theta_0)|_{T_0}$.
Equivalently,
\begin{equation*}
\Ker\bigl(Dq_{a,d}^{(1)}(\theta_0)|_{T_0}\bigr) \subset \Ker\bigl(Dq_{\cQ}^{(1)}(\theta_0)|_{T_0}\bigr).
\end{equation*}
Hence all hypotheses are satisfied, and the conclusion follows.

\end{proof}

\begin{remark}
\Cref{cor:example-fits} shows that the worked example falls under the general strict-loss theory.
Frozen stationary fiber weights are not needed; the parameter dependence of the fiber weights is absorbed by the product-rule formula in \cref{lem:coord-fiber} and the general selected-coordinate criterion.

\end{remark}

\section{Strict-loss criteria}
\label{sec:branching}

\begin{definition}[Hidden fiber over a visible state]
Fix $m<r$ and a visible state $y\in\cS_m$.
The corresponding depth-$r$ hidden fiber is
\begin{equation*}
F_y:=\Pi_{r,m}^{-1}(y)\subset\cS_r.
\end{equation*}

\end{definition}

\begin{lemma}[Smooth conditional fiber weights]
\label{lem:fiber-weights} Assume exact depth $r$ at $\theta_0$ together with \cref{ass:regularity}, and fix $m<r$. Let $y\in\cS_m$ and $z\in F_y$. Then, for $\theta$ near $\theta_0$, the stationary conditional weight
\begin{equation}
\alpha_z(\theta):=P_{\theta,\mathrm{stat}}(Z_t^{(r)}=z\mid Z_t^{(m)}=y)
\end{equation}
is well defined and $C^1$ in $\theta$.
More explicitly,
\begin{equation}
\alpha_z(\theta)=\frac{\pi_\theta(z)}{\sum_{w\in F_y}\pi_\theta(w)},
\end{equation}
where $\pi_\theta$ is the stationary law of the depth-$r$ chain.

\end{lemma}

\begin{proof}
Under stationarity, $Z_t^{(m)}$ is the truncation of $Z_t^{(r)}$, so the joint event $\{Z_t^{(r)}=z,\,Z_t^{(m)}=y\}$ equals $\{Z_t^{(r)}=z\}$ whenever $z\in F_y$.
The denominator is positive by \cref{ass:regularity}\textup{(iv)}, and $C^1$ smoothness follows from \cref{lem:stationary} together with the relative-affine convention introduced earlier.

\end{proof}

\begin{theorem}[Sharp blockwise characterization of strict coarse-depth loss]
\label{thm:strict-loss-characterization} Fix $m<r$ and assume exact depth $r$ at $\theta_0$ together with \cref{ass:regularity}. Let $T\subset T_{\theta_0}\Theta$ be a tangent block and let $T_0\subset T$ be a linear subspace of dimension $p\ge1$. Assume
\begin{equation}
\rank\bigl(Dq_{\cQ}^{(r)}(\theta_0)|_{T_0}\bigr)=p.
\end{equation}
Then the following are equivalent:
\begin{enumerate}[label=\textup{(\roman*)}]
\item the depth-$m$ window is not first-order locally sufficient relative to depth $r$ on $T_0$,
\item there exists a nonzero vector $h\in T_0$ such that
\begin{equation*}
Dq_{\cQ}^{(m)}(\theta_0)h=0, \qquad Dq_{\cQ}^{(r)}(\theta_0)h\neq0,
\end{equation*}
\item \begin{equation*}
\Ker\bigl(Dq_{\cQ}^{(m)}(\theta_0)|_{T_0}\bigr)\neq\{0\},
\end{equation*}
\item \begin{equation*}
\rank\bigl(Dq_{\cQ}^{(m)}(\theta_0)|_{T_0}\bigr)<p.
\end{equation*}
\end{enumerate}

\end{theorem}

\begin{proof}
Set
\begin{equation*}
L_m:=Dq_{\cQ}^{(m)}(\theta_0)|_{T_0}, \qquad L_r:=Dq_{\cQ}^{(r)}(\theta_0)|_{T_0}.
\end{equation*}
The rank assumption implies that $L_r$ is injective on the $p$-dimensional space $T_0$, hence $\Ker(L_r)=\{0\}$.
\smallskip \noindent\textup{(i)}$\Rightarrow$\textup{(ii)}. Failure of first-order local sufficiency means $\Ker(L_m)\not\subset\Ker(L_r)$, so there exists $h\in\Ker(L_m)$ with $h\notin\Ker(L_r)$. Then $L_mh=0$ and $L_rh\neq0$, and in particular $h\neq0$.
\smallskip \noindent\textup{(ii)}$\Rightarrow$\textup{(iii)}. A nonzero vector $h$ with $L_mh=0$ lies in $\Ker(L_m)$, so that kernel is nontrivial.
\smallskip \noindent\textup{(iii)}$\Rightarrow$\textup{(iv)}. If $\Ker(L_m)$ is nontrivial, then rank--nullity gives
\begin{equation*}
\rank(L_m)=p-\dim\Ker(L_m)<p.
\end{equation*}
\smallskip \noindent\textup{(iv)}$\Rightarrow$\textup{(iii)}. If $\rank(L_m)<p$, then rank--nullity gives
\begin{equation*}
\dim\Ker(L_m)=p-\rank(L_m)>0,
\end{equation*}
so $\Ker(L_m)$ is nontrivial.
\smallskip \noindent\textup{(iii)}$\Rightarrow$\textup{(i)}. Choose $0\neq h\in\Ker(L_m)$. Since $L_r$ is injective, $L_rh\neq0$, so $h\notin\Ker(L_r)$. Therefore $\Ker(L_m)\not\subset\Ker(L_r)$, which is exactly failure of first-order local sufficiency.

\end{proof}

\begin{remark}
\Cref{thm:strict-loss-characterization} is the sharp deterministic statement on a tangent block where the depth-$r$ derivative is injective: strict coarse-depth loss is equivalent to coarse rank loss.
Thus the later certification criteria provide practical ways to verify the hypotheses of this characterization.

\end{remark}

\begin{theorem}[Intrinsic quotient-space characterization of strict coarse-depth loss]
\label{thm:quotient-strict-loss} Fix $m<r$ and assume exact depth $r$ at $\theta_0$ together with \cref{ass:regularity}. Let $T\subset T_{\theta_0}\Theta$ be a tangent block. Set
\begin{equation}
L_m:=Dq_{\cQ}^{(m)}(\theta_0)|_T, \qquad L_r:=Dq_{\cQ}^{(r)}(\theta_0)|_T, \qquad K_r:=\Ker(L_r).
\end{equation}
Let $\pi:T\to T/K_r$ be the quotient map, and let
\begin{equation*}
\widetilde L_m, \widetilde L_r:T/K_r\to \RR^N
\end{equation*}
denote the unique linear maps induced by $L_m$ and $L_r$, where $N$ is any ambient coordinate dimension containing the relevant images.
Then the following hold:
\begin{enumerate}[label=\textup{(\roman*)}]
\item $\widetilde L_r$ is injective,
\item the depth-$m$ window is first-order locally sufficient relative to depth $r$ on $T$ if and only if
\begin{equation}
\Ker(\widetilde L_m)=\{0\},
\end{equation}
\item the depth-$m$ window is not first-order locally sufficient relative to depth $r$ on $T$ if and only if
\begin{equation}
\rank(\widetilde L_m)<\dim(T/K_r),
\end{equation}
\item the depth-$m$ window is not first-order locally sufficient relative to depth $r$ on $T$ if and only if
\begin{equation}
\rank\bigl(Dq_{\cQ}^{(m)}(\theta_0)|_T\bigr)<\rank\bigl(Dq_{\cQ}^{(r)}(\theta_0)|_T\bigr).
\end{equation}
\end{enumerate}

\end{theorem}

\begin{proof}
Since $K_r=\Ker(L_r)$, the universal property of the quotient gives a unique linear map
\begin{equation*}
\widetilde L_r:T/K_r\to\im(L_r)
\end{equation*}
with $L_r=\widetilde L_r\circ\pi$.
If $\widetilde L_r([h])=0$, then $L_rh=0$, so $h\in K_r$ and therefore $[h]=0$ in $T/K_r$.
Thus $\widetilde L_r$ is injective, proving \textup{(i)}.
Because exact depth implies the factorization statement in \cref{cor:factorization}, one has
\begin{equation*}
\Ker(L_m)\supset K_r.
\end{equation*}
Hence $L_m$ also factors uniquely through the quotient, yielding
\begin{equation*}
L_m=\widetilde L_m\circ\pi.
\end{equation*}
I now prove the equivalences.
For \textup{(ii)}, first-order local sufficiency on $T$ means exactly
\begin{equation*}
\Ker(L_m)\subset\Ker(L_r)=K_r.
\end{equation*}
Since already $K_r\subset\Ker(L_m)$, this is equivalent to
\begin{equation*}
\Ker(L_m)=K_r.
\end{equation*}
Under the quotient correspondence,
\begin{equation*}
\Ker(\widetilde L_m)=\pi(\Ker(L_m))=\Ker(L_m)/K_r.
\end{equation*}
Therefore $\Ker(\widetilde L_m)=\{0\}$ if and only if $\Ker(L_m)=K_r$, proving \textup{(ii)}.
For \textup{(iii)}, by \textup{(i)} the space $T/K_r$ has dimension
\begin{equation*}
\dim(T/K_r)=\rank(L_r).
\end{equation*}
By rank--nullity applied to $\widetilde L_m:T/K_r\to\RR^N$, the kernel of $\widetilde L_m$ is nontrivial if and only if
\begin{equation*}
\rank(\widetilde L_m)<\dim(T/K_r).
\end{equation*}
Combining this with \textup{(ii)} proves \textup{(iii)}.
For \textup{(iv)}, since $L_m=\widetilde L_m\circ\pi$ and $\pi$ is surjective, one has
\begin{equation*}
\im(L_m)=\im(\widetilde L_m), \qquad \text{hence} \qquad \rank(L_m)=\rank(\widetilde L_m).
\end{equation*}
Likewise, because $\widetilde L_r$ is injective on $T/K_r$,
\begin{equation*}
\rank(L_r)=\dim(T/K_r).
\end{equation*}
Substituting these identities into \textup{(iii)} gives exactly
\begin{equation*}
\rank(L_m)<\rank(L_r).
\end{equation*}
This proves \textup{(iv)}.

\end{proof}

\begin{remark}
\Cref{thm:quotient-strict-loss} removes the auxiliary injectivity hypothesis from \cref{thm:strict-loss-characterization} without enlarging the conclusion beyond what the exact-depth factorization justifies.
On an arbitrary tangent block $T$, the only directions discarded are those already invisible at depth $r$.

\end{remark}

\begin{definition}[Coordinate map]
\label{def:coordinate-map} Fix a finite family
\begin{equation*}
I=\{(y_j,a_j):1\le j\le s\},
\end{equation*}
where each $y_j\in\cS_m$ is a visible state and each $a_j$ is an admissible appended edge from the terminal vertex of $y_j$.
Writing $y_{j,a}:=U_m(y_j,a_j)$, define
\begin{equation}
\Phi_I^{(m)}(\theta):=\bigl(q_{y_j,y_{j,a}}^{(m)}(\theta)\bigr)_{j=1}^s\in\RR^s.
\end{equation}

\end{definition}

\begin{lemma}[Equivalent factorization through selected coordinates]
\label{lem:selected-factorization} Fix $m<r$ and a tangent subspace $T_0\subset T_{\theta_0}\Theta$. Let
\begin{equation}
M_I:=D\Phi_I^{(m)}(\theta_0)|_{T_0}:T_0\to\RR^s.
\end{equation}
The following are equivalent:
\begin{enumerate}[label=\textup{(\roman*)}]
\item \begin{equation}
\Ker(M_I)\subset\Ker\bigl(Dq_{\cQ}^{(m)}(\theta_0)|_{T_0}\bigr),
\end{equation}
\item there exists a linear map
\begin{equation*}
B:\im(M_I)\to\im\bigl(Dq_{\cQ}^{(m)}(\theta_0)|_{T_0}\bigr)
\end{equation*}
such that
\begin{equation*}
Dq_{\cQ}^{(m)}(\theta_0)|_{T_0}=B\circ M_I.
\end{equation*}
\end{enumerate}

\end{lemma}

\begin{proof}
The implication \textup{(ii)}$\Rightarrow$\textup{(i)} is immediate.
Conversely, assume \textup{(i)}.
For any $\nu\in\im(M_I)$ choose $h\in T_0$ such that $\nu=M_Ih$ and define
\begin{equation*}
B(\nu):=Dq_{\cQ}^{(m)}(\theta_0)h.
\end{equation*}
If also $\nu=M_Ih'$, then $h-h'\in\Ker(M_I)\subset\Ker(Dq_{\cQ}^{(m)}(\theta_0)|_{T_0})$, so $Dq_{\cQ}^{(m)}(\theta_0)h=Dq_{\cQ}^{(m)}(\theta_0)h'$.
Hence $B$ is well defined, and linearity is immediate.

\end{proof}

\begin{theorem}[Coordinate criterion for strict coarse-depth loss]
\label{thm:selected-criterion} Fix $m<r$ and assume exact depth $r$ at $\theta_0$ together with \cref{ass:regularity}. Let $T\subset T_{\theta_0}\Theta$ be a tangent block and let $T_0\subset T$ be a linear subspace of dimension $p\ge1$. Assume
\begin{equation}
\rank\bigl(Dq_{\cQ}^{(r)}(\theta_0)|_{T_0}\bigr)=p.
\end{equation}
Let $I$ be a finite family of visible state / appended-edge pairs as above, and let $M_I:=D\Phi_I^{(m)}(\theta_0)|_{T_0}$.
If
\begin{enumerate}[label=\textup{(\alph*)}]
\item \begin{equation*}
\Ker(M_I)\subset\Ker\bigl(Dq_{\cQ}^{(m)}(\theta_0)|_{T_0}\bigr),
\end{equation*}
\item \begin{equation*}
\rank(M_I)<p,
\end{equation*}
\end{enumerate}
then the depth-$m$ window is not first-order locally sufficient relative to depth $r$ on $T_0$.

\end{theorem}

\begin{proof}
By \cref{lem:selected-factorization}, assumption \textup{(a)} implies that the full coarse derivative factors linearly through $M_I$.
Hence
\begin{equation*}
\rank\bigl(Dq_{\cQ}^{(m)}(\theta_0)|_{T_0}\bigr)\le \rank(M_I)<p.
\end{equation*}
The conclusion therefore follows from \cref{thm:strict-loss-characterization}.

\end{proof}

\begin{theorem}[Global selected-coordinate criterion via quotient rank]
\label{thm:selected-criterion-global} Fix $m<r$ and assume exact depth $r$ at $\theta_0$ together with \cref{ass:regularity}. Let $T\subset T_{\theta_0}\Theta$ be a tangent block. Let $I$ be a finite family of visible state--appended-edge pairs as in \cref{def:coordinate-map}, and let
\begin{equation}
M_I:=D\Phi_I^{(m)}(\theta_0)|_T:T\to\RR^s.
\end{equation}
Assume
\begin{enumerate}[label=\textup{(\alph*)}]
\item \begin{equation*}
\Ker(M_I)\subset\Ker\bigl(Dq_{\cQ}^{(m)}(\theta_0)|_T\bigr),
\end{equation*}
\item \begin{equation*}
\rank(M_I)<\rank\bigl(Dq_{\cQ}^{(r)}(\theta_0)|_T\bigr).
\end{equation*}
\end{enumerate}
Then the depth-$m$ window is not first-order locally sufficient relative to depth $r$ on $T$.

\end{theorem}

\begin{proof}
Set
\begin{equation*}
L_m:=Dq_{\cQ}^{(m)}(\theta_0)|_T, \qquad L_r:=Dq_{\cQ}^{(r)}(\theta_0)|_T.
\end{equation*}
By assumption \textup{(a)} and \cref{lem:selected-factorization}, there exists a linear map
\begin{equation*}
B:\im(M_I)\to\im(L_m)
\end{equation*}
such that
\begin{equation*}
L_m=B\circ M_I.
\end{equation*}
Therefore
\begin{equation*}
\im(L_m)\subset B(\im(M_I)),
\end{equation*}
which implies
\begin{equation*}
\rank(L_m)\le \rank(M_I).
\end{equation*}
Combining this with assumption \textup{(b)} yields
\begin{equation*}
\rank(L_m)<\rank(L_r).
\end{equation*}
The conclusion follows from \cref{thm:quotient-strict-loss}\textup{(iv)}.

\end{proof}

\begin{remark}
\Cref{thm:selected-criterion-global} extends \cref{thm:selected-criterion} in applicability without changing the form of the conclusion.
When the fine-depth derivative is already injective on a chosen subspace $T_0$, the global theorem restricted to $T_0$ recovers the earlier subspace-based criterion.

\end{remark}

\begin{corollary}[Global single-coordinate branching criterion]
\label{cor:single-coordinate-branching-global} Fix $m<r$ and assume exact depth $r$ at $\theta_0$ together with \cref{ass:regularity}. Let $T\subset T_{\theta_0}\Theta$ be a tangent block. Suppose there exist a visible state $y\in\cS_m$ and an admissible appended edge $a$ such that, writing $y_a:=U_m(y,a)$,
\begin{enumerate}[label=\textup{(\alph*)}]
\item \begin{equation}
\Ker\bigl(Dq_{y,y_a}^{(m)}(\theta_0)|_T\bigr) \subset \Ker\bigl(Dq_{\cQ}^{(m)}(\theta_0)|_T\bigr),
\end{equation}
\item \begin{equation}
Dq_{y,y_a}^{(m)}(\theta_0)|_T\neq0,
\end{equation}
\item \begin{equation}
\rank\bigl(Dq_{\cQ}^{(r)}(\theta_0)|_T\bigr)\ge2.
\end{equation}
\end{enumerate}
Then the depth-$m$ window is not first-order locally sufficient relative to depth $r$ on $T$.

\end{corollary}

\begin{proof}
Apply \cref{thm:selected-criterion-global} with $I=\{(y,a)\}$.
Then
\begin{equation*}
M_I=Dq_{y,y_a}^{(m)}(\theta_0)|_T:T\to\RR.
\end{equation*}
Assumption \textup{(a)} gives the factorization hypothesis.
Assumption \textup{(b)} implies that $M_I$ is nonzero, hence
\begin{equation*}
\rank(M_I)=1.
\end{equation*}
By assumption \textup{(c)},
\begin{equation*}
1=\rank(M_I)<\rank\bigl(Dq_{\cQ}^{(r)}(\theta_0)|_T\bigr).
\end{equation*}
All hypotheses of \cref{thm:selected-criterion-global} are therefore satisfied.

\end{proof}

\begin{lemma}[Coordinatewise stationary-fiber representation]
\label{lem:coord-fiber} Fix a pair $(y,a)$ consisting of a visible state $y\in\cS_m$ and an admissible appended edge $a$ from the terminal vertex of $y$, and write $y_a:=U_m(y,a)$. For each $z\in F_y$ define
\begin{equation}
\alpha_z(\theta):=P_{\theta,\mathrm{stat}}(Z_t^{(r)}=z\mid Z_t^{(m)}=y), \qquad \zeta_z(\theta):=q_{z,U_r(z,a)}^{(r)}(\theta).
\end{equation}
Then, for every $\theta$ near $\theta_0$,
\begin{equation}
q_{y,y_a}^{(m)}(\theta)=\sum_{z\in F_y}\alpha_z(\theta)\zeta_z(\theta).
\end{equation}
Consequently,
\begin{equation}
Dq_{y,y_a}^{(m)}(\theta_0)|_{T_0} = \sum_{z\in F_y}\alpha_z(\theta_0)D\zeta_z(\theta_0)|_{T_0} +
\sum_{z\in F_y}\zeta_z(\theta_0)D\alpha_z(\theta_0)|_{T_0}.
\end{equation}

\end{lemma}

\begin{proof}
Because the model has exact depth $r$, $Z^{(m)}$ is the pathwise truncation of $Z^{(r)}$.
Conditional on $Z_t^{(m)}=y$, the hidden state $Z_t^{(r)}$ lies in $F_y$.
Conditioning on the hidden state and using the law of total probability gives the first identity, and differentiating the finite sum of products gives the second.

\end{proof}

\begin{corollary}[Single-coordinate branching special case]
\label{cor:single-coordinate-branching} Fix $m<r$ and assume exact depth $r$ at $\theta_0$ together with \cref{ass:regularity}. Let $T\subset T_{\theta_0}\Theta$ be a tangent block and let $T_0\subset T$ be two-dimensional. Suppose there exist a visible state $y\in\cS_m$ and an admissible appended edge $a$ such that, writing $y_a:=U_m(y,a)$,
\begin{enumerate}[label=\textup{(\alph*)}]
\item \begin{equation}
\Ker\bigl(Dq_{y,y_a}^{(m)}(\theta_0)|_{T_0}\bigr) \subset \Ker\bigl(Dq_{\cQ}^{(m)}(\theta_0)|_{T_0}\bigr),
\end{equation}
\item \begin{equation}
Dq_{y,y_a}^{(m)}(\theta_0)|_{T_0}\neq0,
\end{equation}
\item \begin{equation}
\rank\bigl(Dq_{\cQ}^{(r)}(\theta_0)|_{T_0}\bigr)=2.
\end{equation}
\end{enumerate}
Then the depth-$m$ window is not first-order locally sufficient relative to depth $r$ on $T_0$.

\end{corollary}

\begin{proof}
This is the special case of \cref{thm:selected-criterion} in which $I$ consists of a single coordinate.
Since $T_0$ is two-dimensional and the selected covector is nonzero, its rank is $1<2$, so assumption \textup{(b)} of \cref{thm:selected-criterion} is automatic.

\end{proof}

\begin{corollary}[Explicit product-rule matrix criterion]
\label{cor:explicit-matrix} In the setting of \cref{thm:selected-criterion}, let $L_I:T_0\to\RR^s$ be the linear map whose $j$th coordinate covector is
\begin{equation*}
\sum_{z\in F_{y_j}}\alpha_{j,z}(\theta_0)D\zeta_{j,z}(\theta_0)|_{T_0} + \sum_{z\in F_{y_j}}\zeta_{j,z}(\theta_0)D\alpha_{j,z}(\theta_0)|_{T_0},
\end{equation*}
where $\alpha_{j,z}$ and $\zeta_{j,z}$ are defined as in \cref{lem:coord-fiber} for the pair $(y_j,a_j)$.
Then
\begin{equation}
L_I=M_I.
\end{equation}
In particular, if
\begin{equation}
\Ker(L_I)\subset\Ker\bigl(Dq_{\cQ}^{(m)}(\theta_0)|_{T_0}\bigr), \qquad \rank(L_I)<p, \qquad \rank\bigl(Dq_{\cQ}^{(r)}(\theta_0)|_{T_0}\bigr)=p,
\end{equation}
then the depth-$m$ window is not first-order locally sufficient relative to depth $r$ on $T_0$.

\end{corollary}

\begin{proof}
The identity $L_I=M_I$ follows coordinatewise from \cref{lem:coord-fiber}.
The conclusion is then exactly \cref{thm:selected-criterion}.

\end{proof}

\begin{theorem}[Minimal informative window equals exact depth under global coordinate-rank loss]
\label{thm:minimal-global} Assume exact depth $r$ at $\theta_0$ together with \cref{ass:regularity}, and let $T\subset T_{\theta_0}\Theta$ be a tangent block. Assume:
\begin{enumerate}[label=\textup{(\alph*)}]
\item \begin{equation}
\rank\bigl(Dq_{\cQ}^{(r)}(\theta_0)|_T\bigr)=\dim T,
\end{equation}
\item for every $m<r$ there exists a finite family
\begin{equation}
I^{(m)}=\{(y_j^{(m)},a_j^{(m)}):1\le j\le s_m\}
\end{equation}
of visible state--appended-edge pairs at depth $m$ such that, with
\begin{equation}
M_{I^{(m)}}:=D\Phi_{I^{(m)}}^{(m)}(\theta_0)|_T:T\to\RR^{s_m},
\end{equation}
one has
\begin{equation*}
\Ker\bigl(M_{I^{(m)}}\bigr)\subset\Ker\bigl(Dq_{\cQ}^{(m)}(\theta_0)|_T\bigr) \qquad\text{and}\qquad \rank\bigl(M_{I^{(m)}}\bigr)<\rank\bigl(Dq_{\cQ}^{(r)}(\theta_0)|_T\bigr).
\end{equation*}
\end{enumerate}
Then
\begin{equation}
m_*(T,\theta_0)=r.
\end{equation}

\end{theorem}

\begin{proof}
By assumption \textup{(a)}, the restricted derivative $Dq_{\cQ}^{(r)}(\theta_0)|_T$ has full column rank $\dim T$.
Therefore depth $r$ attains the defining rank threshold for $m_*(T,\theta_0)$, and so
\begin{equation*}
m_*(T,\theta_0)\le r.
\end{equation*}
Fix $m<r$.
By assumption \textup{(b)}, there exists a family $I^{(m)}$ satisfying the displayed kernel inclusion and strict rank inequality.
Applying \cref{thm:selected-criterion-global} with this family yields that the depth-$m$ window is not first-order locally sufficient relative to depth $r$ on $T$.
By \cref{thm:quotient-strict-loss}\textup{(iv)}, this is equivalent to
\begin{equation*}
\rank\bigl(Dq_{\cQ}^{(m)}(\theta_0)|_T\bigr)<\rank\bigl(Dq_{\cQ}^{(r)}(\theta_0)|_T\bigr).
\end{equation*}
Using assumption \textup{(a)} once more gives
\begin{equation*}
\rank\bigl(Dq_{\cQ}^{(m)}(\theta_0)|_T\bigr)<\dim T.
\end{equation*}
Thus no depth $m<r$ has full column rank on $T$.
Since the argument applies to every $m<r$, no smaller depth attains the rank threshold defining $m_*(T,\theta_0)$.
Combined with $m_*(T,\theta_0)\le r$, this proves
\begin{equation*}
m_*(T,\theta_0)=r.
\end{equation*}

\end{proof}

\begin{remark}
\Cref{thm:minimal-global} strengthens \cref{cor:minimal} conceptually by formulating the strict-loss tests directly on the whole tangent block $T$ rather than on auxiliary subspaces $T_0^{(m)}$.

\end{remark}

\begin{corollary}[Minimal informative window equals exact depth under blockwise strict loss]
\label{cor:minimal} Assume exact depth $r$ at $\theta_0$ together with \cref{ass:regularity}, and let $T\subset T_{\theta_0}\Theta$ be a tangent block. Assume:
\begin{enumerate}[label=\textup{(\alph*)}]
\item \begin{equation}
\rank\bigl(Dq_{\cQ}^{(r)}(\theta_0)|_T\bigr)=\dim T,
\end{equation}
\item for every $m<r$ there exists a subspace $T_0^{(m)}\subset T$ such that the hypotheses of \cref{thm:strict-loss-characterization} hold on $T_0^{(m)}$ and
\begin{equation}
\rank\bigl(Dq_{\cQ}^{(m)}(\theta_0)|_{T_0^{(m)}}\bigr)<\dim T_0^{(m)}.
\end{equation}
\end{enumerate}
Then
\begin{equation}
m_*(T,\theta_0)=r.
\end{equation}

\end{corollary}

\begin{proof}
By assumption \textup{(a)}, the restricted derivative $Dq_{\cQ}^{(r)}(\theta_0)|_T$ has full column rank $\dim T$.
Therefore depth $r$ attains the defining rank threshold for $m_*(T,\theta_0)$, and hence
\begin{equation*}
m_*(T,\theta_0)\le r.
\end{equation*}
Fix $m<r$.
By assumption \textup{(b)}, there exists a subspace $T_0^{(m)}\subset T$ such that the hypotheses of \cref{thm:strict-loss-characterization} hold on $T_0^{(m)}$ and
\begin{equation*}
\rank\bigl(Dq_{\cQ}^{(m)}(\theta_0)|_{T_0^{(m)}}\bigr)<\dim T_0^{(m)}.
\end{equation*}
Applying \cref{thm:strict-loss-characterization} on $T_0^{(m)}$ shows that $Dq_{\cQ}^{(m)}(\theta_0)|_{T_0^{(m)}}$ has nontrivial kernel.
Hence there exists $0\neq h\in T_0^{(m)}\subset T$ such that
\begin{equation*}
Dq_{\cQ}^{(m)}(\theta_0)h=0.
\end{equation*}
Therefore $Dq_{\cQ}^{(m)}(\theta_0)|_T$ cannot have full column rank $\dim T$.
Since this holds for every $m<r$, no smaller depth attains full column rank on $T$.
Combined with $m_*(T,\theta_0)\le r$, this proves
\begin{equation*}
m_*(T,\theta_0)=r.
\end{equation*}

\end{proof}

\section{Categorical reformulation}
This optional section records a compact reformulation of the deterministic branching mechanism.
No later proof depends on this section.

\begin{definition}[Visible projection functor, local form]
Fix $m<r$.
The truncation map $\Pi_{r,m}:\cS_r\to\cS_m$ induces a projection from depth-$r$ visible transitions to depth-$m$ visible transitions by summing over hidden fibers with stationary weights.

\end{definition}

\begin{theorem}[Categorical form of the branching criterion]
\label{thm:cat-branching} In the exact-depth regime, the depth-$m$ informative map is obtained from the depth-$r$ informative map by composition with the deterministic truncation of states together with stationary fiber averaging. If a tangent block satisfies the hypotheses of either \cref{thm:selected-criterion} or \cref{thm:selected-criterion-global}, then the induced first-order morphism exhibits strict kernel enlargement at depth $m$ relative to depth $r$ after quotienting by the directions already invisible at depth $r$.

\end{theorem}

\begin{proof}
The state-level part is exactly \cref{prop:projection-general}.
The averaging part is the explicit formula from \cref{cor:factorization}.
The strict kernel enlargement statement follows from \cref{thm:selected-criterion} in the injective-block setting and from \cref{thm:selected-criterion-global} together with \cref{thm:quotient-strict-loss} on an arbitrary tangent block.

\end{proof}

\section{Conditional statistical recovery of the minimal informative window}
\label{sec:estimation}
This section is conditional and included only for completeness.
It records a simple plug-in consistency principle once the relevant derivative estimators and a uniform singular-value gap are already available, but it does not construct such estimators for a concrete class of quiver-valued variable-length Markov chains.
In particular, it is not a statistical treatment of model selection for quiver-valued variable-length Markov chains in full generality.

\begin{assumption}[Plug-in rank recovery setup]
\label{ass:estimation} Fix a tangent block $T\subset T_{\theta_0}\Theta$ of dimension $p$, and fix an integer $M\ge1$ such that $m_*(T,\theta_0)\le M$. For each $1\le m\le M$, assume there is a random matrix estimator
\begin{equation}
\widehat J_{n,m}\to Dq_{\cQ}^{(m)}(\theta_0)|_T
\end{equation}
in probability, entrywise and hence in operator norm, after bases on $T$ and the target coordinate spaces are fixed.
Assume moreover that there exists a constant $\gamma>0$ such that
\begin{equation*}
\sigma_{\min}\bigl(Dq_{\cQ}^{(m_*)}(\theta_0)|_T\bigr)\ge 2\gamma,
\end{equation*}
where $m_*:=m_*(T,\theta_0)$, $p:=\dim T$, and $\sigma_{\min}$ denotes the $p$-th singular value of the restricted derivative, with the convention that this value is $0$ whenever the target dimension is smaller than $p$.
For $m<m_*(T,\theta_0)$ the restricted derivative is rank-deficient.

\end{assumption}

\begin{theorem}[Conditional consistency of the minimal-window estimator]
\label{thm:estimation} Under \cref{ass:estimation}, define
\begin{equation}
\widehat m_n:=\min\Bigl\{1\le m\le M:\sigma_{\min}(\widehat J_{n,m})>\gamma\Bigr\},
\end{equation}
with the convention $\widehat m_n=M+1$ if the set is empty.
Then
\begin{equation}
\widehat m_n\to m_*(T,\theta_0) \qquad\text{in probability.}
\end{equation}

\end{theorem}

\begin{proof}
For $m<m_*$, the matrix $Dq_{\cQ}^{(m)}(\theta_0)|_T$ is rank-deficient, so its smallest singular value is $0$.
By continuity of singular values under operator-norm perturbations and the assumed consistency of $\widehat J_{n,m}$,
\begin{equation*}
\sigma_{\min}(\widehat J_{n,m})\to 0 \qquad\text{in probability.}
\end{equation*}
Hence
\begin{equation*}
\PP\bigl(\sigma_{\min}(\widehat J_{n,m})>\gamma\bigr)\to0 \qquad (m<m_*).
\end{equation*}
At $m=m_*$, the singular-value gap assumption gives
\begin{equation*}
\sigma_{\min}\bigl(Dq_{\cQ}^{(m_*)}(\theta_0)|_T\bigr)\ge2\gamma.
\end{equation*}
Therefore
\begin{equation*}
\PP\bigl(\sigma_{\min}(\widehat J_{n,m_*})>\gamma\bigr)\to1.
\end{equation*}
Combining the finitely many subcritical depths with the critical one shows that, with probability tending to $1$, no depth smaller than $m_*$ crosses the threshold $\gamma$ while depth $m_*$ does.
Hence $\widehat m_n=m_*$ with probability tending to $1$.

\end{proof}

\section{Conditional LAN kernel transfer}
\label{sec:lan}
This section is conditional.
It records how the deterministic kernels identified earlier propagate to Gaussian LAN limits once an additional likelihood-level factorization hypothesis is imposed.
Besides bare LAN for the chosen experiment, one needs a likelihood factorization through the visible informative map together with a nondegeneracy condition on the induced quadratic form along the image of the derivative.
The conclusions below should therefore be read only as transfer statements from the deterministic kernel criteria to the Gaussian shift limit.
Bare LAN alone does not identify the deterministic kernels appearing in the earlier sections.

\begin{remark}
The LAN material below is deliberately separated from the deterministic rank theory.
The deterministic sections prove inclusions and rank comparisons for derivatives of informative maps, whereas the Gaussian statements additionally require an external LAN input and the factorized hypothesis \cref{ass:lan}.

\end{remark}

\begin{assumption}[LAN factorization at depth $\ell$]
\label{ass:lan} Fix a visible depth $\ell\ge1$ and a tangent block $T\subset T_{\theta_0}\Theta$. Assume the experiment generated by $(Z_0^{(\ell)},\dots,Z_n^{(\ell)})$ is LAN at $\theta_0$ on $T$, and that for local perturbations $h_n=h/\sqrt n$ with $h\in T$ the log-likelihood ratio admits the expansion
\begin{equation*}
\log\frac{dP_{\theta_0+h_n}^{(\ell,n)}}{dP_{\theta_0}^{(\ell,n)}} = \langle\Lambda_\ell h,\Delta_{n,\ell}\rangle-\frac12\|\Lambda_\ell h\|^2+o_{P_{\theta_0}}(1),
\end{equation*}
where $\Delta_{n,\ell}\Rightarrow N(0,I_{N_\ell})$, $N_\ell$ is the ambient coordinate dimension of $q_{\cQ}^{(\ell)}$, and
\begin{equation*}
\Lambda_\ell=J_\ell^{1/2}Dq_{\cQ}^{(\ell)}(\theta_0)|_T
\end{equation*}
for some symmetric positive semidefinite matrix $J_\ell$ on the ambient coordinate space of $q_{\cQ}^{(\ell)}$ whose quadratic form is positive definite on $\im(Dq_{\cQ}^{(\ell)}(\theta_0)|_T)$.

\end{assumption}

\begin{proposition}[Conditional LAN at the true depth]
\label{prop:lan-true} Assume exact depth $r$ at $\theta_0$. Suppose, after shrinking to a neighborhood of $\theta_0$ if necessary, that the depth-$r$ chain has fixed finite support, is irreducible, the positive transition coordinates depend smoothly on $\theta$, and the initial law is either stationary or contributes only an $o_{P_{\theta_0}}(1)$ term to the local log-likelihood ratio. If a standard finite-state Markov-chain LAN theorem is invoked under these hypotheses for the chosen parameterization, then the experiment generated by $Z^{(r)}$ is LAN at $\theta_0$ on every fixed tangent block $T\subset T_{\theta_0}\Theta$. This proposition records only bare LAN, not the factorized form required by \cref{ass:lan}.

\end{proposition}

\begin{proof}
By \cref{prop:augmentation}, the observed depth-$r$ process is a finite-state time-homogeneous Markov chain.
The additional hypotheses in the statement are precisely those needed to import a standard LAN theorem for smooth irreducible finite-state Markov chains in the present parameterization, under that external theorem, the claim follows.
See, for example, \cite{billingsley,bickel-ritov} for background on finite-state Markov chains and asymptotic statistical arguments of this type.
No likelihood-factorization statement is claimed here.

\end{proof}

\begin{proposition}[Conditional LAN for coarser windows]
\label{prop:lan-coarse} Fix $m<r$ and assume exact depth $r$ at $\theta_0$. Then $Z^{(m)}$ is a deterministic function of the hidden finite-state Markov chain $Z^{(r)}$. Suppose, in addition, that the projected family satisfies the hypotheses of a standard finite hidden-Markov-model LAN theorem appropriate to this deterministic-emission setting, including the required fixed hidden support, irreducibility/mixing, smooth dependence of the positive hidden transitions on $\theta$, and any domination or identifiability conditions used by the invoked theorem. Then the experiment generated by $Z^{(m)}$ is LAN at $\theta_0$ on every fixed tangent block $T\subset T_{\theta_0}\Theta$. This statement provides only bare LAN for the coarse observation scheme and does not by itself yield the factorized shift representation required in \cref{ass:lan}.

\end{proposition}

\begin{proof}
By \cref{prop:projection-general}, the visible process $Z^{(m)}$ is obtained by applying the deterministic map $\Pi_{r,m}$ coordinatewise to the hidden Markov chain $Z^{(r)}$.
Under the extra hypotheses stated above, one may invoke a finite hidden-Markov-model LAN theorem that covers this deterministic observation mechanism, and the claim then follows.
See, for example, \cite{baum-petrie,bickel-ritov,cappe-moulines-ryden} for hidden-Markov background and asymptotic methodology.
As in \cref{prop:lan-true}, only bare LAN is asserted here.

\end{proof}

\begin{remark}
\Cref{prop:lan-true,prop:lan-coarse} are intentionally phrased as import statements rather than self-contained LAN proofs.
Their role is only to isolate when bare LAN is available, every later Gaussian comparison still relies on the stronger factorized hypothesis \cref{ass:lan}.

\end{remark}

\begin{theorem}[Kernel alignment under the factorized LAN hypothesis]
\label{thm:kernel-align} Assume \cref{ass:lan}. Then
\begin{equation}
\Ker\Lambda_\ell= \Ker\bigl(Dq_{\cQ}^{(\ell)}(\theta_0)|_T\bigr).
\end{equation}
Consequently, for $h,h'\in T$,
\begin{equation}
\Lambda_\ell h=\Lambda_\ell h' \quad\Longleftrightarrow\quad h-h'\in\Ker\bigl(Dq_{\cQ}^{(\ell)}(\theta_0)|_T\bigr).
\end{equation}

\end{theorem}

\begin{proof}
By \cref{ass:lan}, the LAN shift map factors as
\begin{equation*}
\Lambda_\ell=J_\ell^{1/2}Dq_{\cQ}^{(\ell)}(\theta_0)|_T.
\end{equation*}
If $h\in\Ker(Dq_{\cQ}^{(\ell)}(\theta_0)|_T)$, then the right-hand side vanishes, so $h\in\Ker\Lambda_\ell$.
This proves
\begin{equation*}
\Ker\bigl(Dq_{\cQ}^{(\ell)}(\theta_0)|_T\bigr) \subset \Ker\Lambda_\ell.
\end{equation*}
For the reverse inclusion, let $h\in\Ker\Lambda_\ell$ and set
\begin{equation*}
v:=Dq_{\cQ}^{(\ell)}(\theta_0)|_T h\in\im\bigl(Dq_{\cQ}^{(\ell)}(\theta_0)|_T\bigr).
\end{equation*}
Then
\begin{equation*}
0=\|\Lambda_\ell h\|^2 =\|J_\ell^{1/2}v\|^2 =v^\top J_\ell v.
\end{equation*}
By assumption, the quadratic form induced by $J_\ell$ is positive definite on
\begin{equation*}
\im\bigl(Dq_{\cQ}^{(\ell)}(\theta_0)|_T\bigr).
\end{equation*}
Since $v$ lies in that image and $v^\top J_\ell v=0$, it follows that $v=0$.
Hence
\begin{equation*}
Dq_{\cQ}^{(\ell)}(\theta_0)|_T h=0,
\end{equation*}
so $h\in\Ker(Dq_{\cQ}^{(\ell)}(\theta_0)|_T)$.
Therefore
\begin{equation*}
\Ker\Lambda_\ell= \Ker\bigl(Dq_{\cQ}^{(\ell)}(\theta_0)|_T\bigr).
\end{equation*}
The equivalence for pairs $h,h'$ follows by applying this identity to $h-h'$.

\end{proof}

\begin{corollary}[Gaussian loss from deterministic loss]
\label{cor:gaussian-loss} Fix $m<r$ and a tangent block $T\subset T_{\theta_0}\Theta$. Assume \cref{ass:lan} holds at depths $m$ and $r$. If the depth-$m$ window is not first-order locally sufficient relative to depth $r$ on $T$, then
\begin{equation}
\Ker\Lambda_m\not\subset\Ker\Lambda_r.
\end{equation}
In particular, there exists a local direction that is asymptotically invisible in the coarse Gaussian shift but visible in the fine one.

\end{corollary}

\begin{proof}
Failure of first-order local sufficiency means that there exists $h\in T$ such that
\begin{equation*}
Dq_{\cQ}^{(m)}(\theta_0)h=0, \qquad Dq_{\cQ}^{(r)}(\theta_0)h\neq0.
\end{equation*}
By \cref{thm:kernel-align}, this is equivalent to
\begin{equation*}
\Lambda_m h=0, \qquad \Lambda_r h\neq0.
\end{equation*}
Hence $h\in\Ker\Lambda_m\setminus\Ker\Lambda_r$.

\end{proof}

\section{Deterministic synthesis and scope}
\label{sec:synthesis}

\begin{theorem}[Deterministic rank comparison in the two tractable regimes]
\label{thm:main-dichotomy} Let $\theta_0\in\Theta$ and let $T\subset T_{\theta_0}\Theta$ be a tangent block.
\begin{enumerate}[label=\textup{(\roman*)}]
\item Suppose the model is edge-homogeneous near $\theta_0$, satisfies \cref{ass:regularity}, and the representation hypothesis of \cref{thm:homogeneous} holds for the visible depths under consideration. Then all such visible depths have the same first-order rank on $T$.
\item Suppose the model has exact depth $r$ at $\theta_0$ and satisfies \cref{ass:regularity}. Then for every $m<r$,
\begin{equation}
\rank\bigl(Dq_{\cQ}^{(m)}(\theta_0)|_T\bigr) \le \rank\bigl(Dq_{\cQ}^{(r)}(\theta_0)|_T\bigr).
\end{equation}
Moreover, strict coarse-depth loss on $T$ is equivalent to strict rank drop from depth $r$ to depth $m$ on $T$ itself.
If, in addition, the hypotheses of \cref{thm:selected-criterion-global} hold on $T$, or the hypotheses of \cref{thm:selected-criterion} hold on some subspace $T_0\subset T$, then depth $m$ loses a nonzero first-order direction relative to depth $r$ on the corresponding block.
\end{enumerate}

\end{theorem}

\begin{proof}
Part \textup{(i)} is exactly \cref{thm:homogeneous}.
The monotonicity statement in part \textup{(ii)} is \cref{cor:factorization}, the intrinsic quotient-space form of strict loss is \cref{thm:quotient-strict-loss}, the certification mechanisms are provided by \cref{thm:strict-loss-characterization,thm:selected-criterion,thm:selected-criterion-global}, and the strengthened exact-depth recovery statement is \cref{thm:minimal-global}.
The concrete depth-two realization is given by \cref{cor:example-fits}.

\end{proof}

\begin{remark}
The theorem above is a comparison theorem, not an exhaustive classification of all quiver-valued variable-length Markov chains.
The edge-homogeneous and exact-depth regimes isolate two structurally tractable settings in which precise first-order rank statements can be proved.
Models outside these regimes may require different methods.

\end{remark}

\begin{remark}
The theorem above is the deterministic core of the manuscript.
It gives a local dichotomy between exact equality of ranks across visible depths in the edge-homogeneous regime and monotone loss of information under deterministic truncation in the exact-depth regime.

\end{remark}

\begin{remark}
The strongest depth-recovery statement in this manuscript is \emph{not} that exact depth automatically implies $m_*(T,\theta_0)=r$.
Exact depth yields rank monotonicity, while the identities \cref{cor:minimal,thm:minimal-global} require additional strict-loss input, formulated either on auxiliary subspaces or directly on the whole tangent block through selected coordinates.

\end{remark}

\begin{remark}
The statistical and LAN sections are conditional transfer principles.
The plug-in theorem requires derivative estimators and a singular-value gap, while the Gaussian comparison statements require the factorized LAN hypothesis in \cref{ass:lan}.
Bare LAN by itself does not identify the deterministic kernels studied earlier.

\end{remark}


\begin{thebibliography}
{99} \bibitem{baum-petrie} L.
E.
Baum and T.
Petrie, \newblock Statistical inference for probabilistic functions of finite state Markov chains, \newblock \emph{Ann.
Math.
Statist.} \textbf{37} (1966), 1554--1563.
\bibitem{bickel-ritov} P.
J.
Bickel and Y.
Ritov, \newblock Inference in hidden Markov models I: Local asymptotic normality in the stationary case, \newblock in \emph{Theory of Statistics}, de Gruyter, Berlin, 1986.
\bibitem{billingsley} P.
Billingsley, \newblock \emph{Statistical Inference for Markov Processes}, \newblock University of Chicago Press, Chicago, 1961.
\bibitem{buhlmann-modelsel} P.
B\"uhlmann, \newblock Model selection for variable length Markov chains and tuning the context algorithm, \newblock \emph{Ann.
Inst.
Statist.
Math.} \textbf{52} (2000), 287--315.
\bibitem{buhlmann-wyner} P.
B\"uhlmann and A.
J.
Wyner, \newblock Variable length Markov chains, \newblock \emph{Ann.
Statist.} \textbf{27} (1999), 480--513.
\bibitem{cappe-moulines-ryden} O.
Capp\'e, E.
Moulines, and T.
Ryd\'en, \newblock \emph{Inference in Hidden Markov Models}, \newblock Springer, New York, 2005.
\bibitem{machler-vlmc} M.
M\"achler and P.
B\"uhlmann, \newblock Variable length Markov chains: Methodology, computing, and software, \newblock \emph{J.
Comput.
Graph.
Statist.} \textbf{13} (2004), 435--455.
\bibitem{rissanen} J.
Rissanen, \newblock A universal data compression system, \newblock \emph{IEEE Trans.
Inform.
Theory} \textbf{29} (1983), 656--664.
\end{thebibliography}
\end{document}